\documentclass[11pt,a4paper]{article}
\usepackage{amsfonts,amssymb,amsmath,amscd,latexsym,makeidx,theorem,xcolor, stmaryrd}

\author{Antonio De Rosa, Domenico Angelo La Manna}

\title{A 
nonlocal approximation of the Gaussian perimeter: Gamma convergence and Isoperimetric properties}
\date{}

\newtheorem{trm}{Theorem}
\newtheorem{prop}[trm]{Proposition}

\newcommand{\R}[1]{\mathbb{R}^{#1}}
\newcommand{\de}{\partial}
\newcommand{\ve}{\varepsilon}

\newcommand{\M}[1]{\mathcal{#1}}

\newcommand{\abs}[1]{\left\arrowvert{#1}\right\arrowvert}

\newcommand{\haus}{\mathcal{H}}

\newcommand{\fromto}[2]{\colon #1 \longrightarrow #2}
\def\restrict#1{\raise-.5ex\hbox{\ensuremath|}_{#1}}
\newenvironment{proof}{\noindent\emph{Proof.}}{\hfill$\square$\medskip}

\DeclareMathOperator{\diver}{div}
\DeclareMathOperator{\loc}{loc}

\DeclareMathOperator*{\dist}{dist}
\DeclareMathOperator*{\Intm}{\int\!\!\!\!\!\! \rule[2.6pt]{6.5pt}{.4pt}}

\begin{document}
\maketitle

\begin{abstract}
We study a non local approximation of the Gaussian perimeter, proving the Gamma convergence to the local one. Surprisingly, in contrast with the local setting, the halfspace turns out to be a volume constrained stationary point if and only if the boundary hyperplane passes through the origin. In particular, this implies that Ehrhard symmetrization can in general increase the considered non local Gaussian perimeter.
\end{abstract}

\section{Introduction}

The Gaussian isoperimetric inequality says that the halfspace has the smallest Gaussian perimeter among all sets with prescribed Gaussian measure, \cite{Borell}. 
In the Euclidean setting, an increasing interest has been devoted to the study of non local approximations of the perimeter and their isoperimetric shapes, since the pioneering work of Caffarelli, Roquejoffre and Savin, \cite{CRS}.

%

The aim of this paper is to provide an analogous non local approximation of the Gaussian perimeter, showing the Gamma convergence to the local one. Moreover, we study the isoperimetric properties of this non local functional and observe that, in contrast with the local setting, an halfspace is a volume constrained critical point if and only if it has Gaussian measure $\frac 12$. In particular, we deduce that Ehrhard symmetrization can in general increase the considered non local Gaussian perimeter.

We remark that the non local approximation of the Gaussian perimeter we study is different from the one recently proposed in \cite{NPS}. 
The non local functional we introduce has the advantage of having a more explicit formulation, while has the drawback that the isoperimetric shapes and the Ehrhard symmetrization are not preserved.

Inspired by  \cite{ADPM}, for a measurable set $E\subset\R{n}$, $n\ge 1$, $0<s<1$, and a
connected, open set $\Omega\Subset\R{n}$ with Lipschitz boundary (or
simply $\Omega=(a,b)\Subset\R{}$ if $n=1$), we define the Gaussian, nonlocal
functional
$$\M{J}^\gamma_s(E, \Omega):=\M{J}^{1, \, \gamma}_s(E,\Omega)+\M{J}^{2, \, \gamma}_s(E, \Omega),$$
where
\begin{equation*}
\begin{split}
\M{J}^{1, \, \gamma}_s(E, \Omega)&:=\int_{E\cap \Omega}\int_{E^c\cap \Omega}\frac{\gamma(x, \, y)}{|x-y|^{n+s}} \, dxdy,\\
\M{J}^{2, \, \gamma}_s(E, \Omega)&:=\int_{E\cap \Omega}\int_{E^c\cap
\Omega^c}\frac{\gamma(x, \, y)}{|x-y|^{n+s}} \, dxdy+ \int_{E\cap
\Omega^c}\int_{E^c\cap \Omega}\frac{\gamma(x, \, y)}{|x-y|^{n+s}} \, dxdy,
\end{split}
\end{equation*}
and
\[
\gamma \fromto{\R{n} \times \R{n}}{\R{+}},\ \gamma(x, \, y) = \exp\left(- \frac{1}{4} \left(\abs{x}^2 + \abs{y}^2\right)\right).
\]
When $\Omega$ coincides with the whole space, we just write $\M{J}^{\gamma}_s(E)$. 

In \cite{ADPM}, Amborsio, De Philippis and Martinazzi have studied the Euclidean version of it, namely $\M{J}_s = \M{J}_s^1 + \M{J}_s^2$, and
\begin{equation*}
\begin{split}
\M{J}^1_s(E, \Omega)&:=\int_{E\cap \Omega}\int_{E^c\cap \Omega}\frac{1}{|x-y|^{n+s}} \, dxdy,\\
\M{J}^2_s(E, \Omega)&:=\int_{E\cap \Omega}\int_{E^c\cap
\Omega^c}\frac{1}{|x-y|^{n+s}} \, dxdy+ \int_{E\cap
\Omega^c}\int_{E^c\cap \Omega}\frac{1}{|x-y|^{n+s}} \, dxdy.
\end{split}
\end{equation*}
The authors point out that $\M{J}_s(E,\Omega)$ can be thought of as a fractional perimeter of $E$ in $\Omega$, and they show the $\Gamma$-convergence of $(1-s)\M{J}_s(\cdot,\Omega)$ to $\omega_{n-1}P(\cdot,\Omega)$  as $s\to 1^-$, where $\omega_{n-1}$ is the volume of the unit ball in $\R{n-1}$, $P(E, \, \Omega):= \haus^{n-1}(\M{F} E \cap \Omega)$ is the Euclidean perimeter, $\haus^\alpha$ denotes the classical $\alpha$-Hausdorff measure and $\M{F} E$ the reduced boundary of $E$. 
Moreover, they prove the convergence of any sequence $\{E_i\}$ of local minimizers for
$\M{J}_{s_i}(\cdot, \Omega)$ to a local minimizer for $P(\cdot,\Omega)$, see \cite[Theorem 3]{ADPM}.

The first aim of this paper is to generalize \cite[Theorem 3]{ADPM} to the Gaussian case, thus building a relation between the functional $\M{J}^{\gamma}_s$ and the Gaussian perimeter
\[
P^\gamma(E, \, \Omega):= \int_{\M{F} E \cap \Omega} e^{-\frac{1}{2} \abs{x}^2} \, d\haus^{n-1}(x).
\]
The second goal is to investigate whether the halfspaces are volume constrained critical points of $\M{J}^{\gamma}_s$. This turns out to be true if and only if the boundary hyperplane passes through the origin.

The paper is divided in four Sections. In Section 2 we prove the $\Gamma-$convergence of the functional $\M{J}^{\gamma}_s$ to $P^\gamma$. 
In Section 3 we compute the first and second variation of $\M{J}^{\gamma}_s$
(for the local framework see \cite{BBJ} or \cite{L}). 
In Section 4 we prove that halfspaces are volume constrained stationary points for $\M{J}^{\gamma}_s$ if and only if their Gaussian volume is $\frac 12$.

\section{The Gamma-convergence }
In this section we extend \cite[Theorem 3]{ADPM} to the Gaussian case. Namely, we show:
\begin{trm}[Convergence of local minimizers]\label{trm4}
Assume that $s_i\uparrow 1$, $E_i$ are local minimizers of
$\M{J}^{\gamma}_{s_i}(\cdot, \Omega)$, and $\chi_{E_i}\to\chi_E$ in
$L^1_{\loc}(\R{n})$. Then
\begin{equation}\label{Genergybou}
\limsup_{i\to \infty}(1-s_i)\M{J}^{\gamma}_{s_i}(E_i,\Omega')< + \infty \qquad \forall\Omega'\Subset\Omega,
\end{equation}
$E$ is a local minimizer of $P^{\gamma}(\cdot,\Omega)$ and
$(1-s_i)\M{J}^{\gamma}_{s_i}(E_i,\Omega')\to \omega_{n-1}P(E,\Omega')$ whenever
$\Omega'\Subset\Omega$ and $P(E,\partial\Omega')=0$.
\end{trm} 
The proof of Theorem \ref{trm4} is almost identical to the Euclidean one for \cite[Theorem 3]{ADPM}. We limit our study to the parts which differ from it.  In particular we will prove the following two propositions. Let $\omega_k$ denote the volume of the unit ball in $\R{k}$ for $k\ge 1$, and set $\omega_0:=1$.
\begin{prop}\label{GammaLiminfProp}
For every measurable set $E\subset\R{n}$ we have
\begin{equation}\label{Gliminf}
\Gamma-\liminf_{s\uparrow 1}(1-s)\M{J}^{1, \, \gamma}_{s}(E,\Omega) \geq \omega_{n-1}P^\gamma(E,\Omega)
\end{equation}
w.r.t. the $L^1_{\rm
loc}$ convergence of the corresponding characteristic functions in
$\R{n}$, i.e. 
$$\liminf_{i\to \infty}(1-s_i)\M{J}^{1, \, \gamma}_{s_i}(E_i,\Omega)\geq \omega_{n-1}P^\gamma(E,\Omega)\qquad \text{whenever } \chi_{E_i}\to \chi_E \text{ in } L^1_{\loc}(\R{n}),\;s_i\uparrow 1.$$
\end{prop}

\begin{prop}\label{GammaLimsupProp}
For every measurable set $E\subset\R{n}$ we have
\begin{equation}\label{Glimsup}
\Gamma-\limsup_{s\uparrow 1}(1-s)\M{J}^\gamma_{s}(E,\Omega) \leq \omega_{n-1}P^\gamma(E,\Omega)
\end{equation}
w.r.t. the $L^1_{\rm
loc}$ convergence of the corresponding characteristic functions in
$\R{n}$. Inequality \eqref{Glimsup} means that for every measurable set $E$ and sequence $s_i\uparrow 1$ there exists a sequence $E_i$ with $\chi_{E_i}\to \chi_E$ in $L^1_{\loc}(\R{n})$ such that
$$\limsup_{i\to\infty}(1-s_i)\M{J}^\gamma_{s_i}(E_i,\Omega)\leq \omega_{n-1} P^\gamma(E,\Omega).$$
\end{prop}

In these two propositions lurk the main differences between the Gaussian case and the Euclidean case. Once we have proved Proposition \ref{GammaLiminfProp} and Proposition \ref{GammaLimsupProp}, we are done: the proof of Theorem \ref{trm4} is completely identical to the proof of \cite[Theorem 3]{ADPM}, with the only forethought of adding a $\gamma$-superscript in every considered functional, and remembering the simple inequality $\gamma(x, \, y) \le 1$. 

\medskip

We will use the following notation: we write $x\in\R{n}$ as $(x',x_n)$ with $x'\in\R{n-1}$
and $x_n\in\R{}$; we denote by $H$ the halfspace $\{x:\ x_n\leq 0\}$ and
by $Q=(-1/2,1/2)^n$ the canonical unit cube; we denote by $B_r(x)$ the ball of radius $r$ centered at $x$ and,
unless otherwise specified, $B_r:= B_r(0)$;  for every $h\in \R{n}$ and function $u$ defined on
$U\subset \R{n}$ we set $\tau_h u(x):= u(x+h)$ for all $x\in U-h$. For the definition and basic properties of the perimeter $P(E,\Omega)$ in the
sense of De Giorgi, we refer to \cite{AFP,giu}.

\subsection{Proof of proposition \ref{GammaLiminfProp}}
We denote by ${\mathcal C}$ the family of all $n$-cubes in $\R{n}$
$$
{\mathcal C}:=\left\{R(x+rQ):\ x\in\R{n},\,\,r>0,\,\,R\in
SO(n)\right\}.
$$

Let $s_i\uparrow 1$ and sets $E_i\subset\R{n}$ with $\chi_{E_i}\to \chi_E$
in $L^1_{\loc}(\R{n})$ as $i\to \infty$ be given. We  need to show the inequality 
\begin{equation}\label{liminf}
\liminf_{i\to\infty}(1-s_i)\M{J}^{1, \, \gamma}_{s_i}(E_i,\Omega)\ge \omega_{n-1} P^{\gamma}(E,\Omega).
\end{equation}
We can assume that the left-hand side of \eqref{liminf} is finite,
otherwise the inequality is trivial.  We choose an arbitrary $\Omega' \Subset \Omega$, and find a positive constant $c_0=c_0(\Omega')$ so that $c_0 \le \gamma(x, \, y),\ \forall \, x, \, y \in \Omega'$. Then we easily obtain the inequality
\[
c_0 \limsup_i \M{J}^1_{s_i} (1-s_i) (E_i, \, \Omega'_i) \le \lim_i (1- s_i)\M{J}^{1, \, \gamma}_{s_i}(E_i, \, \Omega') < +\infty .
\] 
By \cite[Theorem 1]{ADPM} and the arbitrariness of $\Omega'$, we conclude that $E$ has locally finite perimeter.  We shall denote by $\mu$ its perimeter measure, i.e.
$\mu(A)=|D\chi_E|(A)$ for any Borel set $A\subset\Omega$, and we
shall use the following property of sets of finite perimeter: for
$\mu$-a.e. $x_0\in\Omega$ there exists $R_{x_0}\in SO(n)$ such that
$(E-x_0)/r$ locally converge in measure to $R_{x_0}H$ as $r\to 0$.
In addition,
\begin{equation}\label{densitycube}
\lim_{r\to 0}\frac{\mu(x_0+rR_{x_0}Q)}{r^{n-1}}=1,\quad \text{for  $\mu$-a.e. } x_0.
\end{equation}
Indeed this property holds for every $x_0\in \M{F} E$, see \cite[Theorem 3.59(b)]{AFP}.

Now, given a cube $C\in{\mathcal C}$ contained in $\Omega$, we set
$$\alpha_i(C):=(1-s_i)\M{J}^{1, \, \gamma}_{s_i}(E_i,C), \qquad \mbox{and} \qquad \alpha (C):=\liminf_{i\to\infty}\alpha_i(C).$$
Moreover, we define $C_r(x_0):=x_0+rR_{x_0}Q$, where $R_{x_0}$ is as in
\eqref{densitycube}, and the measure
\[
\nu(E) = \int_{E} e^{-\frac{1}{2} \abs{x}^2} \, d\mu(x),\qquad \mbox{for every } E \mbox{ Borel set}.
\]
We claim that  for $\mu$-a.e. $x_0\in \R{n}$ it holds
\begin{equation}\label{eq1}
\omega_{n-1} \le \liminf_{r\to 0}\frac{\alpha(C_r(x_0))}{\nu( C_r(x_0))}.
\end{equation}
If the claim is true, then we observe that for all $\ve>0$ the family
$$\M{A}:=\Big\{C_r(x_0)\subset\Omega\;:\; \omega_{n-1} \nu(C_r(x_0)) \le (1+\ve) \alpha(C_r(x_0))   \Big\}$$
is a fine covering of $\mu$-almost all of $\Omega$. By a suitable
variant of Vitali's theorem (see \cite{M}), we can extract a
countable subfamily of disjoint cubes $\{C_j\subset\Omega:j\in J\}$ such that
$\nu\big(\Omega\setminus\bigcup\limits_{j\in J} C_j\big)=0$, whence
\begin{equation*}
\begin{split}
\omega_{n-1} P^\gamma(E,\Omega)&=\omega_{n-1} \nu\Big(\bigcup_{j\in J}C_j\Big)=\omega_{n-1}\sum_{j\in J}\nu(C_j)\leq (1+\ve)\sum_{j\in J}\alpha(C_j)\le (1+\ve)\liminf_{i\to\infty}\sum_{j\in J}\alpha_i(C_j)\\
&\leq (1+\ve)\liminf_{i\to\infty}(1-s_i)\M{J}^{1,\gamma}_{s_i}(E_i,\Omega).
\end{split}
\end{equation*}
Since $\ve>0$ is arbitrary, we get the $\Gamma-\liminf$ estimate.

We now prove the inequality in \eqref{eq1} at any point $x_0$ such
that $(E-x_0)/r$ converges locally in measure as $r\to 0$ to
$R_{x_0}H$ and \eqref{densitycube} holds. Because of
\eqref{densitycube} and the continuity of the exponential, we know that 
\[
\lim_{r \to 0} \Intm_{C_r(x_0)} e^{- \frac{1}{2} \abs{x}^2 } \, d\mu(x) = e^{-\frac{1}{2} \abs{x_0}^2}.
\]
Thus we just need to show the inequality
\begin{equation}\label{eq11}
\liminf_{r\to 0}\frac{\alpha(C_r(x_0))}{r^{n-1}}\geq \omega_{n-1} e^{-\frac{1}{2} \abs{x_0}^2}.
\end{equation}
Since from now on $x_0$ is fixed, we assume $R_{x_0}=I$, so that the limit
hyperplane is $H$ and the cubes $C_r(x_0)$ are the standard ones
$x_0+rQ$. Let us choose a sequence $r_k\to 0$ such that
$$\liminf_{r\to 0}\frac{\alpha(C_r(x_0))}{r^{n-1}}=
\lim_{k\to \infty} \frac{\alpha(C_{r_k}(x_0))}{r_k^{n-1}}.$$ For $k>0$
we can choose $i(k)$ large enough that the following conditions hold:
\begin{equation*}
\left\{
\begin{aligned}
&\alpha_{i(k)}(C_{r_k}(x_0))\leq \alpha(C_{r_k}(x_0))+r_k^n,\\
&r_k^{1-s_{i(k)}}\ge 1-\frac{1}{k},\\
&\Intm_{C_{r_k}(x_0)}|\chi_{E_{i(k)}}-\chi_E|dx<\frac{1}{k}.
\end{aligned}
\right.
\end{equation*}
We observe that, although $\M{J}^{1, \, \gamma}_s$ does not enjoy the nice scaling properties of $\M{J}^{1}_s$, it still satisfies the equality
\[
\M{J}^{1, \, \gamma}_s (E, \, C_r(x_0))= r^{n - s}\M{J}^{1, \, \gamma_{x_0, \, r}}_s((E- x_0)/r, \, Q),
\]
where we have set
\[
\gamma_{x_0, \, r}(x, \, y) = \exp \left(- \frac{1}{4} \left( \abs{x_0 + rx}^2 + \abs{x_0 + r y}^2  \right)\right).
\]
In particular, for $r$ sufficiently small, and thus for every $r_k$ with $k$ sufficiently big, the following inequality holds:
\[
\abs{\gamma_{x_0, \, r}(x, \, y) - e^{- \frac{1}{2} \abs{x_0}^2}} \le 4 r.
\]
Then we infer
\begin{equation*}
\begin{split}
\frac{\alpha(C_{r_k}(x_0))}{r_k^{n-1}}&\ge \frac{\alpha_{i(k)}(C_{r_k}(x_0))}{r_k^{n-1}}-r_k=\frac{(1-s_{i(k)})\M{J}^{1, \, \gamma_{x_0, \, r_k}}_{s_{i(k)}}((E_{{i(k)}}-x_0)/r_k,Q)r_k^{n-s_{i(k)}}}{r_k^{n-1}}-r_k\\
&\ge
\Big(1-\frac{1}{k}\Big)(1-s_{i(k)})\M{J}^{1, \, \gamma_{x_0, \, r_k}}_{s_{i(k)}}((E_{{i(k)}}-x_0)/r_k,Q)-r_k \\
&\ge
\Big(1-\frac{1}{k}\Big)(1-s_{i(k)})\M{J}^{1}_{s_{i(k)}}((E_{{i(k)}}-x_0)/r_k,Q)(e^{- \frac{1}{2} \abs{x_0^2}} - r_k)-r_k,
\end{split}
\end{equation*}
i.e.
$$\lim_{k\to\infty} \frac{\alpha(C_{r_k}(x_0))}{r_k^{n-1}}\ge e^{- \frac{1}{2} \abs{x_0}^2}
\liminf_{k\to\infty}(1-s_{i(k)})\M{J}^1_{s_{i(k)}}((E_{{i(k)}}-x_0)/r_k,Q).$$
Since we have
$$\lim_{k\to\infty}\int_{Q}|\chi_{(E_{{i(k)}}-x_0)/r_k}-\chi_{(E-x_0)/r_k}|dx=0,$$
and
$$\lim_{k\to\infty}\int_{Q}|\chi_{(E-x_0)/r_k}-\chi_H|dx=0,$$
it follows that $(E_{{i(k)}}-x_0)/r_k\to H$ in $L^1(Q)$. 
If we define
\begin{equation}\label{defGamman}
\Gamma_n:=\inf\Big\{\liminf_{s\uparrow 1}(1-s)\M{J}^1_s(E_s,Q)\;\Big|\;
\chi_{E_s}\to \chi_H\text{ in }L^1(Q)\Big\},
\end{equation}
it has been proved in \cite[Lemmata 7, 11, 12]{ADPM} that $\Gamma_n=\omega_{n-1}$.
Hence we conclude the claimed inequality \eqref{eq11}.

\subsection{Proof of proposition \ref{GammaLimsupProp}}
As in \cite{ADPM}, it is enough to prove the $\Gamma-\limsup$ inequality for the
collection $\M{B}$ of polyhedra $\Pi$  of finite perimeter which
satisfy $P(\Pi,\partial\Omega)=0$. $\M{B}$ is dense in
energy, i.e. such that for every set $E$ of finite perimeter there
exists $E_k\in\M{B}$ with $\chi_{E_k}\to \chi_E$ in
$L^1_{\loc}(\R{n})$ as $k\to\infty$ and
$\limsup_kP^\gamma(E_k,\Omega)=P^\gamma(E,\Omega)$. We recall that a polyhedron $\Pi$ is in the class $\M{B}$ if and only if
$$\lim_{\delta\to 0}P(\Pi,\Omega^+_\delta\cup \Omega^-_\delta)=0,\quad \mbox{
or equivalently} \quad \lim_{\delta\to 0}P^\gamma(\Pi,\Omega^+_\delta\cup \Omega^-_\delta)=0, $$
where we have set
\begin{equation}\label{omegadelta}
\begin{split}
\Omega^+_\delta:=\{x\in \Omega^c\;|\;d(x,\Omega)<\delta\}, \qquad \Omega^-_\delta:=\{x\in \Omega\;|\;d(x,\Omega^c)<\delta\}.
\end{split}
\end{equation}

We are going to prove that for a polyhedron $\Pi\subset\R{n}$ there holds
\begin{equation}\label{poly}
\limsup_{s\uparrow 1}(1-s)\M{J}^\gamma_s(\Pi,\Omega)\leq
\Gamma_n^* P^\gamma(\Pi,\Omega)+ 2\Gamma_n^* \lim_{\delta\to 0} P^\gamma(\Pi,\Omega^+_\delta\cup \Omega^-_\delta),
\end{equation}
where
\begin{equation}\label{gamman*}
\Gamma_n^*:= \limsup_{s\uparrow 1}(1-s)\M{J}^1_s(H,Q).
\end{equation}
Again, as in \cite[Lemmata 7, 11, 12]{ADPM} we have the equality $\Gamma_n^* = \omega_{n-1}$. We shall divide the proof into two main steps.

 \emph{Step 1.} We first estimate $\M{J}_s^{1, \, \gamma}(\Pi, \Omega)$. For a fixed $\ve>0$ set
\begin{equation*}
(\de\Pi)_\ve:=\{x\in\Omega\;|\; d(x,\de\Pi)<\ve\},\quad (\de\Pi)_\ve^-:=(\de\Pi)_\ve \cap \Pi.
\end{equation*}
We can find $N_\ve$ disjoint cubes $Q^\ve_i\subset\Omega$, $1\leq
i\le N_\ve$, of side length $\ve$ satisfying the following
properties:
\begin{itemize}
\item[(i)]
if $\tilde Q_i^\ve$ denotes the dilation of $Q_i^\ve$ by a factor
$(1+\ve)$, then each cube $\tilde Q_i^\ve$ intersects exactly one
face $\Sigma$ of $\de\Pi$, its barycenter belongs to $\Sigma$ and
each of its sides is either parallel or orthogonal to $\Sigma$;
\item[(ii)]  $\M{H}^{n-1}\left(((\de\Pi)\cap \Omega)\setminus \bigcup_{i=1}^{N_\ve}Q_i^\ve
\right)=|P(\Pi,\Omega)-N_\ve\ve^{n-1} |\to 0$ as $\ve \to 0$.
\end{itemize}
Property (ii), combined with the continuity of the exponential and the property of measures, easily implies 
\begin{equation}\label{expdecay}
\abs{P^\gamma(\Pi, \Omega) - \ve^{n-1}   \sum_{i=1}^{N_\ve}  e^{-\frac{1}{2} \abs{x_i^\ve}^2 } } \to 0 \mbox{ as } \ve \to 0,
\end{equation}
where we have set by $x_i^\ve$ the center of the cubes $Q^\epsilon_i$.
For $x\in\R{n}$ set
$$I_s(x):=\int_{\Pi^c\cap \Omega}\frac{e^{-\frac{1}{4} \abs{y}^2}}{|x-y|^{n+s}} \, dy.$$
We consider several cases.

\medskip

\noindent\emph{Case 1:} $x\in (\Pi\cap \Omega)\setminus(\de\Pi)_\ve^-$. Then for $y\in \Pi^c\cap \Omega$ we have $|x-y|\geq \ve$, hence
$$I_s(x)\leq \int_{(B_\ve(x))^c} \frac{1}{|x-y|^{n+s}}dy=n\omega_n\int_\ve^\infty \frac{1}{\rho^{s+1}}d\rho=\frac{n\omega_{n}}{s\ve^s},$$
since $n\omega_n=\M{H}^{n-1}(S^{n-1})$. Therefore
\begin{equation}\label{case1}
\int_{(\Pi\cap \Omega)\setminus(\de\Pi)_\ve^-}  I_s(x) e^{- \frac{1}{4} \abs{x}^2} \, dx\leq \frac{n\omega_{n} }{s\ve^s} \int_{\Pi\cap \Omega} e^{- \frac{1}{4} \abs{x}^2} \, dx.
\end{equation}

\noindent\emph{Case 2:} $x\in
(\de\Pi)_\ve^-\setminus\bigcup_{i=1}^{N_\ve}Q_i^\ve$. Then
\begin{equation}\label{case2.0}
I_s(x)\leq\int_{(B_{d(x,\Pi^c \cap
\Omega)}(x))^c}\frac{1}{|x-y|^{n+s}}dy=
n\omega_{n}\int_{d(x,\Pi^c\cap\Omega)}^\infty\frac{1}{\rho^{s-1}}d\rho=\frac{n\omega_{n}}{s[d(x,\Pi^c\cap\Omega)]^s}.
\end{equation}
Now write $(\de\Pi)\cap \Omega=\bigcup_{j=1}^J\Sigma_j$, where each $\Sigma_j$ is the intersection of a face of $\de\Pi$ with $\Omega$, and define
$$(\de\Pi)^-_{\ve,j}:=\{x\in (\de\Pi)^-_\ve:\dist(x,\Pi^c\cap\Omega)=\dist(x,\Sigma_j)\}.$$
Clearly $(\de\Pi)^-_\ve=\bigcup_{j=1}^J (\de\Pi)^-_{\ve,j}$. Moreover we have
$$(\de\Pi)^-_{\ve,j}\subset\{x+t\nu:x\in \Sigma_{\ve,j},\, t\in (0,\ve),\, \nu\text{ is the interior unit normal to }\Sigma_{\ve,j}\}, $$
and $\Sigma_{\ve,j}$ is the set of points $x$ belonging to the same hyperplane as $\Sigma_j$ and with $\dist(x,\Sigma_j)\le \ve$. Clearly $\M{H}^{n-1}(\Sigma_{\ve,j})\le \M{H}^{n-1}(\Sigma_j)+C\ve$ as $\ve\to 0$. Then from \eqref{case2.0} we infer
\begin{equation}\label{case2}
\begin{split}
\int_{(\de\Pi)_\ve^-\setminus\bigcup_{i=1}^{N_\ve}Q_i^\ve}I_s(x) e^{- \frac{1}{4} \abs{x}^2} \, dx&\leq \frac{n\omega_{n}}{s}\sum_{j=1}^J\int_{(\de\Pi)_{\ve,j}^-\setminus\bigcup_{i=1}^{N_\ve}Q_i^\ve}\frac{1}{[d(x,\Pi^c)]^s} \, dx\\
&\leq \frac{n\omega_{n}}{s}\sum_{j=1}^J\int_{(\de\Pi)_{\ve,j}^-\setminus\bigcup_{i=1}^{N_\ve}Q_i^\ve}\frac{1}{[d(x,\Sigma_{\ve,j})]^s} \, dx\\
&\leq\frac{n\omega_{n}}{s}\sum_{j=1}^J\int_{(\Sigma_{\ve,j})\setminus\bigcup_{i=1}^{N_\ve}Q_i^\ve}\bigg(\int_0^\ve \frac{dt}{t^s}\bigg)\, d\M{H}^{n-1}\\
&= \frac{n\omega_{n}\ve^{1-s}}{s(1-s)}\M{H}^{n-1}\left(\bigg(\bigcup_{j=1}^J\Sigma_{\ve,j}\bigg)\setminus\bigcup_{i=1}^{N_\ve}Q_i^\ve\right)=\frac{\ve^{1-s}o(1)}{s(1-s)},
\end{split}
\end{equation}
with error $o(1)\to 0$ as $\ve\to 0$ and independent of $s$.

\medskip

\noindent\emph{Case 3:} $x\in \Pi\cap \bigcup_{i=1}^{N_\ve}Q_i^\ve$. In this case we write
\begin{equation*}
\begin{split}
I_s(x)&=\int_{(\Pi^c\cap \Omega)\cap \{y:|x-y|\geq \ve^2\}}\frac{e^{- \frac{1}{4} \abs{y}^2}}{|x-y|^{n+s}} \, dy +\int_{(\Pi^c\cap \Omega)\cap \{y:|x-y|< \ve^2\}}\frac{e^{- \frac{1}{4} \abs{y}^2}}{|x-y|^{n+s}} \, dy\\
&=:I_s^1(x)+I_s^2(x).
\end{split}
\end{equation*}
Then, similar to the case 1,
$$I_s^1(x)\leq n\omega_{n}\int_{\ve^2}^\infty\frac{1}{\rho^{s+1}} \, d\rho=\frac{n\omega_{n}}{s\ve^{2s}},$$
hence (since all cubes are contained in $\Omega$)
\begin{equation}\label{case3.1}
\int_{\Pi\cap \bigcup_{i=1}^{N_\ve}Q_i^\ve}I_s^1(x) e^{-\frac{1}{4} \abs{x}^2 } \, dx\leq
\frac{n\omega_{n}}{s\ve^{2s}} \int_\Omega e^{- \frac{1}{4} \abs{x}^2} \, dx.
\end{equation}
As for $I_s^2(x)$ observe that if $x\in Q_i^\ve$ and $|x-y|\leq
\ve^2$, then $y\in \tilde Q_i^\ve$, where $\tilde Q_i^\ve$ is the
cube obtained by dilating $Q_i^\ve$ by a factor $1+\ve$ (hence the
side length of $\tilde Q_i^\ve$ is $\ve+\ve^2$). Then
\begin{equation}\label{case3.2}
\begin{split}
\int_{\Pi\cap \bigcup_{i=1}^{N_\ve}Q_i^\ve}I_s^2(x) e^{- \frac{1}{4} \abs{x}^2} \, dx 
&\le\sum_{i=1}^{N_\ve}\int_{\Pi\cap Q_i^\ve}\int_{\Pi^c\cap \tilde Q_i^\ve}\frac{e^{- \frac{1}{4} \left(\abs{x}^2 + \abs{y}^2\right)}}{|x-y|^{n+s}} \, dydx \\
&\le\sum_{i=1}^{N_\ve}\int_{\Pi\cap \tilde Q_i^\ve}\int_{\Pi^c\cap \tilde Q_i^\ve}\frac{e^{- \frac{1}{4} \left(\abs{x}^2 + \abs{y}^2\right)}}{|x-y|^{n+s}} \, dydx \\
&\le \left(\sum_{i=1}^{N_\ve} e^{- \frac{1}{2} \abs{x^\ve_i}^2} \right) \M{J}^1_s(H,(\ve+\ve^2) Q)(1 + \ve^2) \\
&= \underbrace{\left( \ve^{n-1 } \sum_{i=1}^{N_\ve} e^{- \frac{1}{2} \abs{x^\ve_i}^2} \right)}_{= P^\gamma(\Pi, \, \Omega) + o(1)} \ve^{1-s}(1+\ve)^{n-s}\M{J}^1_s(H,Q)(1 + \ve^2),
\end{split}
\end{equation}
where in the last identity we used the scaling property 
\begin{equation}\label{scale}
\M{J}^i_s(\lambda
E,\lambda\Omega)=\lambda^{n-s}\M{J}^i_s(E,\Omega)\qquad \text{for
}\lambda>0, \; i=1,2.
\end{equation}
Keeping $\ve>0$ fixed, letting $s$ go to $1$ and putting \eqref{case1}-\eqref{case3.2} together, we infer
\begin{equation*}
\begin{split}
\limsup_{s\uparrow 1}(1-s)\M{J}^{1, \, \gamma}_s(\Pi,\Omega) \leq o(1)+\Gamma_n^*P^{\gamma}(\Pi,\Omega) = o(1) + \omega_{n-1} P^{\gamma}(\Pi, \, \Omega),
\end{split}
\end{equation*}
with error $o(1)\to 0$ as $\ve\to 0$ uniformly in $s$. Since $\ve>0$
is arbitrary, we conclude
\begin{equation}\label{limsupJ1}
\limsup_{s\uparrow 1}(1-s)\M{J}^{1, \, \gamma}_s(\Pi,\Omega) \leq \omega_{n-1} P^{\gamma}(\Pi,\Omega).
\end{equation}

\medskip

\noindent\emph{Step 2.} It now remains to estimate $\M{J}_s^{2,\gamma}$. Let us start by considering the term
$$\int_{\Pi\cap \Omega}\int_{\Pi^c\cap \Omega^c}\frac{e^{- \frac{1}{4} \left(\abs{x}^2 + \abs{y}^2\right)}}{|x-y|^{n+s}} \, dydx.$$

\noindent\emph{Case 1:} $x\in \Pi\cap (\Omega\setminus\Omega_\delta^-)$. Then for
$y\in \Pi^c\cap\Omega^c$ we have $|x-y|\ge \delta$, whence
$$I(x):=\int_{\Pi^c\cap \Omega^c}\frac{e^{ - \frac{1}{4} \abs{y}^2 }}{|x-y|^{n+s}} \, dy\leq
n\omega_{n}\int_\delta^\infty\frac{d\rho}{\rho^{1+s}}=\frac{n\omega_{n}}{s\delta^s}.$$

\noindent\emph{Case 2:} $x\in \Pi\cap \Omega_\delta^-$. In this
case, using the same argument of case $1$ for $y\in
\Pi^c\cap(\Omega^c\setminus\Omega_\delta^+)$, we have
\begin{equation*}
\begin{split}
I(x)=\int_{\Pi^c\cap\Omega_\delta^+}\frac{e^{- \frac{1}{4} \abs{y}^2}}{|x-y|^{n+s}} \, dy +
\int_{\Pi^c\cap(\Omega^c\setminus\Omega_\delta^+)}\frac{e^{- \frac{1}{4} \abs{y}^2}}{|x-y|^{n+s}} \, dy\le
\int_{\Pi^c\cap\Omega_\delta^+}\frac{e^{- \frac{1}{4} \abs{y}^2}}{|x-y|^{n+s}} \, dy +\frac{n\omega_{n}}{s\delta^s}.
\end{split}
\end{equation*}
Therefore
\begin{equation*}
\begin{split}
\int_{\Pi\cap \Omega}\int_{\Pi^c\cap \Omega^c}\frac{e^{- \frac{1}{4} \left(\abs{x}^2 + \abs{y}^2\right)}}
{|x-y|^{n+s}} \, dy dx &\le\frac{2n\omega_{n}}{s\delta^s} \int_\Omega e^{- \frac{1}{4} \abs{x}^2} \, dx+
\int_{\Pi\cap\Omega_\delta^-}\int_{\Pi^c\cap\Omega_\delta^+}\frac{e^{- \frac{1}{4} \left(\abs{x}^2 + \abs{y}^2\right)}}{|x-y|^{n+s}} \, dy dx\\
&\le\frac{2n\omega_{n}}{s\delta^s} \int_\Omega e^{- \frac{1}{4} \abs{x}^2} \, dx + \int_{\Pi\cap(\Omega_\delta^-\cup \Omega_\delta^+)}
\int_{\Pi^c\cap(\Omega_\delta^-\cup \Omega_\delta^+)}\frac{e^{- \frac{1}{4} \left(\abs{x}^2 + \abs{y}^2\right)}}{|x-y|^{n+s}} \, dydx.
\end{split}
\end{equation*}
An obvious similar estimate can be obtained by swapping $\Pi$ and
$\Pi^c$, finally yielding
\begin{equation*}
\begin{split}
\M{J}_s^{2,\gamma}(\Pi,\Omega)&\le \frac{4n\omega_{n}}{s\delta^s} \int_\Omega e^{- \frac{1}{4} \abs{x}^2} \, dx  + 2 \int_{\Pi\cap(\Omega_\delta^-\cup \Omega_\delta^+)}\int_{\Pi^c\cap(\Omega_\delta^-\cup \Omega_\delta^+)}\frac{e^{- \frac{1}{4} \left(\abs{x}^2 + \abs{y}^2\right)}}{|x-y|^{n+s}} \, dydx\\
&=
\frac{4n\omega_{n}}{s\delta^s} \int_\Omega e^{- \frac{1}{4} \abs{x}^2} \, dx   +2\M{J}_s^{1, \, \gamma}(\Pi,\Omega_\delta^-\cup
\Omega_\delta^+).
\end{split}
\end{equation*}
Using inequality \eqref{limsupJ1} applied with the open set $ \Omega_\delta^-\cup \Omega_\delta^+ $, we get
$$\limsup_{s\uparrow 1}(1-s)\M{J}_s^{2, \, \gamma}(\Pi,\Omega)\le 2 \omega_{n-1} P^\gamma(\Pi,\Omega_\delta^-\cup \Omega_\delta^+).$$
Since $\delta>0$ is arbitrary, letting $\delta$ go to zero, we conclude the proof of the Proposition.

\section{First and second variation}
In this section we calculate the first and second variation of $\M{J}^\gamma_{s}(E)$. A similar analysis has been done in \cite{FFMMM} in order to prove the local minimality of the ball for a functional involving nonlocal terms.

First, we fix some notation. Given a vector field  $X \in C_{c}^{2}(\mathbb{R}^n , \mathbb{R}^{n})$, the {\it associated flow} is defined as the  solution of the Cauchy problem
\begin{eqnarray} \label{flusso}
\begin{cases}
\displaystyle\frac{\partial}{\partial t}\Phi(x,t)=X(\Phi (x,t)) \vspace{5pt}
\\
\Phi(x,0)=x.
\end{cases}
\end{eqnarray}
In the following, we shall always write $\Phi_t$ to denote the map $\Phi(\cdot,t)$.
Note that for any given $X$ there exists $\delta>0$ such that, for $t\in[-\delta,\delta]$, the map $\Phi_t$ is a diffeomorphism coinciding with the identity map outside a compact set.

If $E\subset \mathbb{R}^n$ is  measurable, we set $E_t:=\Phi_t(E)$. Denoting by $J\Phi_t$ the $n$-dimensional Jacobian of $\Phi_t$, the first and second derivatives of $J\Phi_t$ are given by
\begin{equation}\label{utili}
\frac{\partial}{\partial t} \restrict{t=0} J\Phi_t =\text{div}X,\qquad\frac{\partial^2}{\partial t^2} \restrict{t=0}   J\Phi_t=\text{div} ((\text{div}X)X).
\end{equation}

Finally, given a sufficiently smooth  bounded open set $E\subset \R{n}$ and a vector field $X$, we recall that the {\it first variation of $\M{J}^\gamma_{s}(E)$} along the vector field $X$ is defined by  
$$
\delta \M{J}^\gamma_{s}(E)[X]:=\frac{d}{dt}\restrict{t=0}\M{J}^\gamma_{s}(E_t),
$$
where $\Phi_t$ is the flow associated with $X$. 
The {\it second variation of $\M{J}^\gamma_{s}(E)$} along the vector field $X$ is defined by
$$
\delta^2 \M{J}^\gamma_{s}(E)[X]= \frac{d^2}{dt^2}\restrict{t=0} \M{J}^\gamma_s (E_t).
$$
If $X$ is a vector field such that $X:=\phi \nu_E$ on $\de E$, where $\nu_E$ denotes the exterior normal to $E$, using the area formula and the divergence theorem, the first  variation of the Gaussian volume can be computed as
\begin{eqnarray} \label{vol1}
\nonumber
\frac{d}{dt} \restrict{t=0}\gamma (E_t)=\frac{d}{dt} \restrict{t=0}\int_{E} J\Phi_t (x)e^{-\frac{|\Phi_t |^2}{2}}dx
= \int_E \left( \diver X - \langle X,  x\rangle \right) e^{-\frac{|x|^2}{2}}dx \\=
\int_E  \diver \left(X e^{-\frac{|x|^2}{2}}\right)dx= \int_{\de E}\phi(x) e^{-\frac{|x|^2}{2}}d\haus^{n-1}_x.
\end{eqnarray}
If $E$ is a set of class $C^2$, given a smooth function $\phi: \de E\rightarrow \mathbb{R}$, it can be extended in a neighborhood $U$ of $\de E$ so that
\begin{equation} \label{hyp}
\frac{\de}{\de \nu} \phi +\phi(H-\langle x, \nu_E\rangle)=0 \quad \text{on}\,\, \de E.    
\end{equation}
The second variation of the Gaussian volume along the vector field $X$ such that $X=\phi \nu_E$ on $\de E$ and $\phi$ satisfies \eqref{hyp}, can be calculated using the divergence theorem and reads as 
\begin{eqnarray*} \label{vol2}
\begin{split}
\frac{d^2}{dt^2} \restrict{t=0}\gamma (E_t)&=\int_E\diver \left(  \diver \left(X e^{-\frac{|x|^2}{2}} \right) X\right)dx=
\int_{\de E}\phi\left(\frac{\de}{\de \nu} \phi +\phi(H-\langle x, \nu_E\rangle )\right) e^{-\frac{|x|^2}{2}}d\haus^{n-1}_x=0
\end{split}
\end{eqnarray*}
Thus, we say that a vector field preserves the Gaussian volume of $E$ if it satisfies
\begin{eqnarray} \label{tecnico}
\int_{\de E}\phi(x) e^{-\frac{|x|^2}{2}}d\haus^{n-1}_x= 0  \quad\text{ and }\quad \frac{\de \phi}{\de \nu}
+\phi(H-\langle x, \nu_E\rangle)=0 \; \text{ for  }\, x\in \de E.
\end{eqnarray}
We note that without this assumptions the expression of the second variation of the Gaussian perimeter even in the local framework is quite complicated, see \cite[Eq. (17)]{BBJ}.

In order to compute the first and second variation of $\M{J}^\gamma_s$, due to the singularity of the Kernel in the integrand, we need to pass through approximations. Thus, given $\delta \in [0,1)$, let $\eta_\delta \in C^\infty_c([0,+\infty),[0,1])$ be such that $\eta_\delta = 1$ on $[0,\delta]\cup [1/\delta, \infty]$, $\eta_\delta=0$
on $[2\delta, 2/\delta]$, $|\eta'|\leq 2/\delta$ on $[0,\infty)$, and $\eta \downarrow 0$ for every $s \in (0,1)$ as $\delta \rightarrow 0^+$. Then we define
$$
K_{\delta} (z):=(1- \eta_\delta(|z|))\frac{1}{|z|^{n+s}}.
$$
Now we show the following theorem.
\begin{trm} \label{variations}
Let $E$ be an open set of class $C^2$ and $X \in C^2_c(\R n,\R n)$ a vector field such that $X = \phi \nu_E$ on $\de E$. Then the first variation of $\M{J}^\gamma_s (E)$ along a vector field $X$ is given by
\begin{equation}\label{varprima}
\de \M{J}^\gamma_s (E)[X]=\int_{\de E}H^{*}_{\de E}(x)(X(x), \nu_E(x))d\haus ^{n-1}_x,
\end{equation}
while the second variation reads as
\begin{equation}\label{varseconda}
\begin{split}
\de ^2  \M{J}^\gamma_s (E)[X]&=\int_ {\de E}\int_{\de E}\frac{e^{- \frac{1}{4} \left(\abs{x}^2 + \abs{y}^2\right)}}{|x-y|^{n+s}}
\left( \abs{\phi (x)-\phi(y) }^2 -\phi ^2 |\nu_E(x)-\nu_E(y)|\right) d\haus^{n-1}_xd\haus^{n-1}_y
\\
&\qquad  +\int_{\partial E}H^*_\delta \left( \phi \left(H-\langle x,\nu_E \rangle \right) +\frac{\de \phi}{\de \nu}\right)\phi d\haus ^{n-1}_x \\
&\qquad -\int_{\de E}\phi^2(x)\int_{\R n}\frac{\langle y- x, \nu_E\rangle }{2} \frac{e^{- \frac{1}{4} \left(\abs{x}^2 + \abs{y}^2\right)}}{|x-y|^{n+s}}dy \, d\haus^{n-1}_x.
\end{split}
\end{equation}
Moreover, if $X$ is volume preserving, then
\begin{equation}\label{varseconda1}
\begin{split}
\de ^2  \M{J}^\gamma_s (E)[X]&=\int_ {\de E}\int_{\de E}\frac{e^{- \frac{1}{4} \left(\abs{x}^2 + \abs{y}^2\right)}}{|x-y|^{n+s}}
\left( \abs{\phi (x)-\phi(y) }^2 -\phi ^2 |\nu_E(x)-\nu_E(y)|\right) d\haus^{n-1}_xd\haus^{n-1}_y
\\
&\qquad -\int_{\de E}\phi^2(x)\int_{\R n}\frac{\langle y-x,\nu_E(x)\rangle}{2} \frac{e^{- \frac{1}{4} \left(\abs{x}^2 + \abs{y}^2\right)}}{|x-y|^{n+s}}dy \, d\haus^{n-1}_x.
\end{split}
\end{equation}
\end{trm}
\begin{proof}
Let us call $\M{J}_s^\delta$ the integral associated to the regularized kernel, namely
$$
\M{J}_s^\delta(E)= \int_{E^c}\int_{E}\frac{e^{- \frac{1}{4} \left(\abs{x}^2 + \abs{y}^2\right)}}{|x-y|^{n+s}+\delta}dxdy.
$$
By the definition of $\Phi_t$, the implicit function theorem gives the existence of $\varepsilon>0$ such that the map $\Phi_t$ is a diffeomorphism for all $t \in [-\delta, \delta]$.
Using the area formula, we compute
$$
\M{J}_s^\delta(E_t)= \int_{E^c}\int_{E}\frac{e^{- \frac{1}{4} \left(\abs{\Phi(x,t)}^2 + \abs{\Phi(y,t)}^2\right)}}{|\Phi(x,t)-\Phi(y,t)|^{n+s}+\delta}J\Phi(x,t)J\Phi(y,t)dxdy
$$
We use the first equation in $\eqref{utili}$ to compute the first variation of $\M{J}_s^\delta$
\begin{equation*}
\begin{split}
\frac{d}{dt}_{|_{t=0}}J_s^\delta (E_t)&= \int_{E^c}\int_{E}e^{- \frac{1}{4} \left(\abs{x}^2 + \abs{y}^2\right)}D_z\left(K_{\delta} (x-y)\right)(X(x)-X(y))dxdy
\\ 
& \, \, \, \, + \int_{E^c}\int_{E}K_{\delta} (x-y)\left( D_x(e^{- \frac{1}{4}\left(\abs{x}^2 + \abs{y}^2\right)})X(x)+ D_y(e^{- \frac{1}{4}\left(\abs{x}^2 + \abs{y}^2\right)})X(y)\right)dxdy
\\ 
& \, \, \, \, + \int_{E^c}\int_{E}K_{\delta} (x-y)e^{- \frac{1}{4} \left(\abs{x}^2 + \abs{y}^2\right)}\left(\diver X(x)+\diver X(y)\right)dxdy\\ 
&  =
 \int_{E^c}\int_{E}\diver_x\left( e^{- \frac{1}{4} \left(\abs{x}^2 + \abs{y}^2\right)}K_{\delta} (x-y)X(x)\right)dxdy\\
& \, \, \, \, +\int_{E^c}\int_{E}\diver_y\left(  e^{- \frac{1}{4} \left(\abs{x}^2 + \abs{y}^2\right)}K_{\delta} (x-y)X(y)\right)dxdy
\\ 
&  =\int_{\de E}\int_{\R n}\left(\chi_{E^c}(y)-\chi_E (y)    \right) K_{\delta} (x-y)e^{- \frac{1}{4} \left(\abs{x}^2 + \abs{y}^2\right)}\langle X(x),\nu_E(x)\rangle dyd\haus^{n-1}_x.
\end{split}
\end{equation*}

Now we compute the second variation of $\M{J}_s^\delta$.
\begin{equation}
\begin{split}
\delta^2 \M{J}_s^\delta[X]&=\frac{d^2}{dt^2}_{|_{t=0}} \int_{E^c}\int_E K_\delta(\Phi(x,t)-\Phi(y,t)) e^{- \frac{1}{4} \left(\abs{\Phi(x,t)}^2 + \abs{\Phi(y,t)}^2\right)}\diver X(x)\diver X(y)dxdy\\ 
&= \int_{E^c}\int_{E} D^2_{xx} \left( K_\delta (x-y) {e^{- \frac{1}{4} \left(\abs{x}^2 + \abs{y}^2\right)}} \right)[X(x),X(x)]
dxdy\\
& \quad+  \int_{E^c}\int_{E}\left \langle D_{x} \left( K_\delta (x-y) {e^{- \frac{1}{4} \left(\abs{x}^2 + \abs{y}^2\right)}} \right),DX(x)X(x)\right \rangle dxdy
\nonumber
\\
& \quad+2\int_{E^c}\int_{E} D^2_{yx} \left( K_\delta (x-y) {e^{- \frac{1}{4} \left(\abs{x}^2 + \abs{y}^2\right)}} \right)[X(x),X(y)]dxdy
\\
& \quad+\int_{E^c}\int_{E} D^2_{yy} \left( K_\delta (x-y) {e^{- \frac{1}{4} \left(\abs{x}^2 + \abs{y}^2\right)}} \right)[X(y),X(y)]dxdy\\
& \quad+\int_{E^c}\int_{E}\left \langle D_{y} \left( K_\delta (x-y) {e^{- \frac{1}{4} \left(\abs{x}^2 + \abs{y}^2\right)}} \right),DX(y)X(y)\right \rangle dxdy\\
& \quad +2\int_{E^c}\int_{E}\left \langle D_{x} \left( K_\delta (x-y) {e^{- \frac{1}{4} \left(\abs{x}^2 + \abs{y}^2\right)}} \right),X(x)\right \rangle
(\diver_x X(x)+\diver_y X(y))dxdy
\\
& \quad +2\int_{E^c}\int_{E}\left \langle D_{y} \left( K_\delta (x-y) {e^{- \frac{1}{4} \left(\abs{x}^2 + \abs{y}^2\right)}} \right),X(y)\right \rangle
(\diver_x X(x)+\diver_y X(y))dxdy
\\
& \quad +\int_{E^c}\int_{E} K_\delta (x-y) {e^{- \frac{1}{4} \left(\abs{x}^2 + \abs{y}^2\right)}}
\big( \diver [X(x)\diver(X(x))]+\diver [X(y)\diver(X(y))]\big)dxdy
\\
& \quad +2\int_{E^c}\int_{E} K_\delta (x-y) {e^{- \frac{1}{4} \left(\abs{x}^2 + \abs{y}^2\right)}}\diver_x X(x)\diver_y X(y)dxdy
\end{split}
\end{equation}

We now use the divergence theorem and exploit the symmetry of $K_\delta$ in order to simplify the above expression as follows
\begin{equation}
\begin{split}
\delta^2 \M{J}_s^\delta&= \int_{E^c}\int_{E}\diver_x \left[X(x) \diver_x \left( K_\delta (x-y) {e^{- \frac{1}{4} \left(\abs{x}^2 + \abs{y}^2\right)}}X(x)\right)\right]dxdy
\\
& \quad + \int_{E^c}\int_{E}\diver_y \left[X(y) \diver_y \left( K_\delta (x-y) {e^{- \frac{1}{4} \left(\abs{x}^2 + \abs{y}^2\right)}}X(y)\right)\right]dxdy
\\ & \quad + \int_{E^c}\int_{E}\diver_x \left[X(x) \diver_y \left( K_\delta (x-y) {e^{- \frac{1}{4} \left(\abs{x}^2 + \abs{y}^2\right)}}X(y)\right)\right]dxdy
\\& \quad + \int_{E^c}\int_{E}\diver_y \left[X(y) \diver_x \left( K_\delta (x-y) {e^{- \frac{1}{4} \left(\abs{x}^2 + \abs{y}^2\right)}}X(x)\right)\right]dxdy\\
&=I_1+I_2+I_3+I_4.
\end{split}
\end{equation}
 Using Fubini and the divergence theorems we have
$$
I_1=\int_{\de E} \langle X(x),\nu_E(x) \rangle \int_{E^c}\diver_x \left( K_\delta (x-y) {e^{- \frac{1}{4} \left(\abs{x}^2 + \abs{y}^2\right)}}X(x)\right)dyd\haus^{n-1}_x
$$
and 
\begin{equation}
\begin{split}
I_3&=\int_{\de E} \langle X(x),\nu_E(x) \rangle \int_{E^c}\diver_x \left( K_\delta (x-y) {e^{- \frac{1}{4} \left(\abs{x}^2 + \abs{y}^2\right)}}X(y)\right)dyd\haus^{n-1}_x
\\
&=-\int_{\de E}\int_{\de E}(K_\delta (x-y) {e^{- \frac{1}{4} \left(\abs{x}^2 + \abs{y}^2\right)}}\langle X(y),\nu_E(y)\rangle \langle X(x),\nu_E(x)\rangle d\haus_x^{n-1}d\haus_y^{n-1}.
\end{split}
\end{equation}
We remark that $I_1$ (resp. $I_3$) has the same expression of $I_2$ (resp. $I_4$) exchanging $x$ and $y$. Using this observation and the symmetry of $K_\delta$, we compute
$$
I_1+I_2=\int_{\de E}\langle X(x),\nu_E(x)\rangle  \int_{\R n}\left(\chi_{E^c} (y)-\chi_E (y)\right)\diver_x \left( K_\delta (x-y) {e^{- \frac{1}{4} \left(\abs{x}^2 + \abs{y}^2\right)}}X(x)\right)dyd\haus^{n-1}_x \\,
$$
and 
$$I_3+I_4=-2\int_{\de E}\int_{\de E}(K_\delta (x-y) {e^{- \frac{1}{4} \left(\abs{x}^2 + \abs{y}^2\right)}}\langle X(y),\nu_E(y)\rangle \langle X(x),\nu_E(x)\rangle  d\haus_x^{n-1}d\haus_y^{n-1}.
$$
Next we write $\diver_x X(x)= \diver_{\nu(x)} X(x)+ \diver_{\tau (x)} X(x)$, where $\diver_{\nu(x)}X(x):= \langle DX[\nu_E(x)], \nu_E(x)\rangle$. Using  Fubini's theorem and the divergence theorem on manifolds, we get
\begin{equation}
\begin{split}
I_1+I_2&= \int_{\R n}\left(\chi_{E^c} (y)-\chi_E (y)\right)\int_{\de E}\langle X(x),\nu_E(x)\rangle  \diver_{\tau (x)}\left( K_\delta (x-y) {e^{- \frac{1}{4} \left(\abs{x}^2 + \abs{y}^2\right)}}X(x)\right)d\haus^{n-1}_xdy\\
&\quad+
\int_{\R n}\left(\chi_{E^c} (y)-\chi_E (y)\right)\int_{\de E}\langle X(x),\nu_E(x)\rangle   \diver_{\nu(x)}  \left( K_\delta (x-y) {e^{- \frac{1}{4} \left(\abs{x}^2 + \abs{y}^2\right)}}X(x)\right)d\haus^{n-1}_xdy
\\
&=
\int_{\R n}\left(\chi_{E^c} (y)-\chi_E (y)\right)\int_{\de E}H(x)\langle X(x),\nu_E(x)\rangle ^2 \left( K_\delta (x-y) {e^{- \frac{1}{4} \left(\abs{x}^2 + \abs{y}^2\right)}}X(x)\right)d\haus^{n-1}_xdy\\
&\quad+
\int_{\R n}\left(\chi_{E^c} (y)-\chi_E (y)\right)\int_{\de E}\langle X(x),\nu_E(x)\rangle   \diver_{\nu(x)} \left( K_\delta (x-y) {e^{- \frac{1}{4} \left(\abs{x}^2 + \abs{y}^2\right)}}X(x)\right)d\haus^{n-1}_xdy \\
&=
\int_{\R n}\left(\chi_{E^c} (y)-\chi_E (y)\right)\int_{\de E}H(x)\phi^2 (x)\left( K_\delta (x-y) {e^{- \frac{1}{4} \left(\abs{x}^2 + \abs{y}^2\right)}}X(x)\right)d\haus^{n-1}_xdy\\
&\quad +
\int_{\R n}\left(\chi_{E^c} (y)-\chi_E (y)\right)\int_{\de E}\phi^2(x)\frac{\de}{\de \nu (x)}\left( K_\delta (x-y) {e^{- \frac{1}{4} \left(\abs{x}^2 + \abs{y}^2\right)}}\right)d\haus^{n-1}_xdy\\
&\quad+
\int_{\R n}\left(\chi_{E^c} (y)-\chi_E (y)\right)\int_{\de E}\phi(x)\frac{\de \phi}{\de \nu(x)}K_\delta (x-y) {e^{- \frac{1}{4} \left(\abs{x}^2 + \abs{y}^2\right)}}d\haus^{n-1}_xdy,
\end{split}
\end{equation}
where we used that $X=\phi \nu_E$ and then $\langle D_\tau f,X \rangle =0$ for every $f \in C^1(\de E)$. 
Regarding the second addend of the above expression, using again Fubini's theorem  and the fact that $D_x K_\delta =-D_y K_\delta$, we get
\begin{equation}
\begin{split}
&\int_{E}\frac{\de}{\de \nu (x)}\left( K_\delta (x-y) {e^{- \frac{1}{4} \left(\abs{x}^2 + \abs{y}^2\right)}}\right)dy\\
&=
-\int_E \left[\left \langle D_y K_\delta (x-y) , \nu_E(x)\right \rangle +\frac{\langle x,\nu_E(x)\rangle }{2} K_\delta (x-y) \right] e^{- \frac{1}{4} \left(\abs{x}^2 + \abs{y}^2\right)}dy\\
&=-\int_{\de E}   K_\delta (x-y) e^{- \frac{1}{4} \left(\abs{x}^2 + \abs{y}^2\right)}\left \langle \nu_E(x),\nu_E(y)\right \rangle d\haus^{n-1}_y
-\int_E \frac{\langle x+y,\nu_E(x)\rangle}{2} K_\delta(x-y) e^{- \frac{1}{4} \left(\abs{x}^2 + \abs{y}^2\right)}dy.
\end{split}
\end{equation}
Finally, thanks to the identity $|\nu_E(x)-\nu_E(y)|^2=2- 2\langle \nu_E(x),\nu_E(y)\rangle$, after some elementary calculations we deduce
\begin{equation}
\begin{split}
\de ^2  \M{J}^\delta_s (E)[X]&=\int_ {\de E}\int_{\de E}e^{- \frac{1}{4} \left(\abs{x}^2 + \abs{y}^2\right)}K_\delta(x-y)
\left( \abs{\phi (x)-\phi(y) }^2 -\phi ^2 |\nu_E(x)-\nu_E(y)|\right) d\haus^{n-1}d\haus^{n-1}\\
&\quad +
\int_{\partial E}H^*_\delta \left( \phi \left(H-\langle x,\nu_E\rangle \right) +\frac{\de \phi}{\de \nu}\right)\phi d\haus ^{n-1}\\
&\quad -\int_{\de E}\phi^2(x)\int_{\R n}\left(\chi_{E^c} (y)-\chi_E (y)\right)\frac{\langle y-x,\nu_E(x)\rangle }{2} K_\delta(x-y) e^{- \frac{1}{4} \left(\abs{x}^2 + \abs{y}^2\right)}dy d\haus^{n-1}_x.
\end{split}
\end{equation}
At this point we just need to show that the first and second variation of $\M{ J}_s^\delta$ converge, respectively, to the first and second variation of $\M{J}^\gamma_s$ as $\delta$ goes to $0$. The proof of this fact is exactly the same as in \cite{FFMMM}.
\end{proof}

Note that, since
$$
\frac{d}{dt}\restrict{t=0}\int_{E_t} e^{-\frac{x^2}{2}}dx=
\int_{\partial E}e^{-\frac{|x|^2}{2}}\langle X(x),\nu_E(x)\rangle d\haus^{n-1},
$$
we have that the flow $\Phi$ associated to $X$ preserves the Gaussian volume if 
$$
\int_{\partial E}e^{-\frac{|x|^2}{2}}\langle X(x),\nu_E(x)\rangle d\haus^{n-1}=0.
$$
Thus, the Euler-Lagrange equation for the problem 
\begin{equation}\label{minimiz}
\min_{|E|=m} \M{J}^\gamma_s(E)
\end{equation}
is 
$$
\int_{\de E} \phi(x) \int_{\R n}\left(\chi_{E^c}(y)-\chi_E (y)    \right) \frac{e^{- \frac{1}{4} \left(\abs{x}^2 + \abs{y}^2\right)}}{|x-y|^{n+s}}dyd\haus^{n-1}_x
= \lambda \int_{\de E} \phi(x) e^{-\frac{|x|^2}{2}}d\haus_x^{n-1}. 
$$
Moreover, if $E$ is a set of class $C^2$, then thanks to the fundamental lemma of the calculus of variations the above equation can be rewritten as
\begin{equation}\label{minimizza}
\int_{\R n}\left(\chi_{E^c}(y)-\chi_E (y)    \right) \frac{e^{- \frac{1}{4} \abs{y}^2}}{|x-y|^{n+s}}dy= \lambda e^{-\frac{\abs{x}^2 }{4}}, \qquad \forall x \in \partial E.
\end{equation}
$E$ is said to be stationary with respect to the non local Gaussian isoperimetric problem, or equivalently a volume constrained critical point, if it satisfies equation \eqref{minimizza}.

\section{Volume constrained stationary shapes}
In this section we prove that, as opposed to the local setting, the only halfspaces which are stationary with respect to the non local Gaussian isoperimetric problem are the ones generated by hyperplanes passing through the origin.
\begin{trm}
We fix $a \in \mathbb R$ and $\omega \in \mathbb S^{n-1}$. If $H_{\omega,a}:= \{x \in \R n: \langle x, \omega\rangle <a\}$ is stationary with respect to the non local Gaussian isoperimetric problem, then $a=0$, or equivalently, 
$$\frac{1}{(2\pi)^{\frac{n}{2}}}\int_{H_{\omega, a}}e^{-\frac{|x|^2}{2}}dx= \frac 12.$$
\end{trm}
\begin{proof}
Up to rotation, we can assume $\omega=e_n$. 
We start observing that, for every $x \in \partial E$, it holds $\langle x, e_n\rangle=a$. This implies that, with the change of coordinate $z=y-x$, if $\langle y, e_n\rangle<a$, then $\langle z, e_n\rangle<0$ and then we can write
\begin{equation}\label{minimizza1}
\begin{split}
\int_{E} \frac{e^{- \frac{1}{4} \abs{y}^2}}{|x-y|^{n+s}}dy&= \int_{\{ y_n<a\}} \frac{e^{- \frac{1}{4} \abs{y}^2}}{|x-y|^{n+s}}dy=\int_{\{ z_n<0\}} \frac{e^{- \frac{1}{4} (\abs{z}^2+\abs{x}^2+2\langle z, x\rangle)}}{|z|^{n+s}}dy\\
&=\int_{\{ z_n>0\}} \frac{e^{- \frac{1}{4} (\abs{z}^2+\abs{x}^2-2\langle z, x\rangle)}}{|z|^{n+s}}dy, \qquad \forall x \in \partial E.
\end{split}
\end{equation}
Analogously, we compute 
\begin{equation}\label{minimizza2}
\int_{E^c} \frac{e^{- \frac{1}{4} \abs{y}^2}}{|x-y|^{n+s}}dy= \int_{\{ y_n>a\}} \frac{e^{- \frac{1}{4} \abs{y}^2}}{|x-y|^{n+s}}dy=\int_{\{ z_n>0\}} \frac{e^{- \frac{1}{4} (\abs{z}^2+\abs{x}^2+2\langle z, x\rangle)}}{|z|^{n+s}}dz \qquad \forall x \in \partial E.
\end{equation}
Plugging equations \eqref{minimizza1}, \eqref{minimizza2} in equation \eqref{minimizza}, we get 
\begin{equation}\label{minimizza3}
\int_{\{ z_n>0\}} \frac{e^{- \frac{1}{4} (\abs{z}^2+\abs{x}^2})}{|z|^{n+s}}\left (e^{ \frac{\langle z, x\rangle}{2}}- e^{- \frac{\langle z, x\rangle}{2}}\right)dz= \lambda e^{-\frac{\abs{x}^2 }{4}}, \qquad \forall x \in \partial E,
\end{equation}
which in turn reads
\begin{equation}\label{minimizza4}
2\int_{\{ z_n>0\}} \frac{e^{- \frac{1}{4} \abs{z}^2}}{|z|^{n+s}}\sinh\left(\frac{\langle z, x\rangle}{2}\right)dz= \lambda , \qquad \forall x \in \partial E.
\end{equation}
We remark that the integral in \eqref{minimizza4} is well defined, since $\lim_{x\to 0}\frac{\sinh(x)}{x}=1$.

We split $x=(x',x_n)$ and we observe that
\begin{equation}\label{minimizza5}
\begin{split}
\sinh\left(\frac{\langle z, x\rangle}{2}\right)&=\sinh\left(\frac{\langle z', x'\rangle+ z_nx_n}{2}\right)\\
&=\sinh\left(\frac{\langle z', x'\rangle}{2}\right)\cosh\left(\frac{ z_nx_n}{2}\right)+\cosh\left(\frac{\langle z', x'\rangle}{2}\right)\sinh\left(\frac{ z_nx_n}{2}\right).
\end{split}
\end{equation}
Plugging \eqref{minimizza5} in \eqref{minimizza4}, we deduce the following equation for every $x \in \partial E$
\begin{equation}\label{minimizza45}
\begin{split}
A+B:=\int_{\{ z_n>0\}}& \frac{e^{- \frac{1}{4} \abs{z}^2}}{|z|^{n+s}}\sinh\left(\frac{\langle z', x'\rangle}{2}\right)\cosh\left(\frac{ z_na}{2}\right)dz\\
&+\int_{\{ z_n>0\}} \frac{e^{- \frac{1}{4} \abs{z}^2}}{|z|^{n+s}}\cosh\left(\frac{\langle z', x'\rangle}{2}\right)\sinh\left(\frac{ z_na}{2}\right)dz= \frac \lambda 2.
\end{split}
\end{equation}
Since $\frac{e^{- \frac{1}{4} (\abs{z'}^2+\abs{z_n}^2})}{(\abs{z'}^2+\abs{z_n}^2)^{\frac{n+s}{2}}}\sinh\left(\frac{\langle z', x'\rangle}{2}\right)\cosh\left(\frac{ z_na}{2}\right)$ is odd in $z'$, we deduce
$$A=\int_0^\infty\int_{\mathbb R^{n-1}} \frac{e^{- \frac{1}{4}( \abs{z'}^2+\abs{z_n}^2)}}{(\abs{z'}^2+\abs{z_n}^2)^{\frac{n+s}{2}}}\sinh\left(\frac{\langle z', x'\rangle}{2}\right)\cosh\left(\frac{ z_na}{2}\right)dz'dz_n=0.$$
Plugging this information in \eqref{minimizza5} and taking the partial derivative in $x_j$, for every $j=1,\dots,n-1$, we deduce
$$
\int_{\{ z_n>0\}} \frac{e^{- \frac{1}{4} \abs{z}^2}}{|z|^{n+s}}\frac{\partial}{ \partial x_j}\left(\cosh\left(\frac{\langle z', x'\rangle}{2}\right)\right)\sinh\left(\frac{ z_na}{2}\right)dz= 0.
$$
Assuming without loss of generality that $j=n-1$ and denoting $x'=(x'',x_{n-1})$, we obtain
\begin{equation}\label{minimizza6}
\begin{split}
C&+D:=\int_{\{ z_n>0\}} \frac{e^{- \frac{1}{4} \abs{z}^2}}{|z|^{n+s}}\frac{\partial}{ \partial x_j}\left(\cosh\left(\frac{\langle z'', x''\rangle}{2}\right)\cosh\left(\frac{\langle z_{n-1}, x_{n-1}\rangle}{2}\right)\right) \sinh\left(\frac{ z_na}{2}\right)dz\\
& +\int_{\{ z_n>0\}} \frac{e^{- \frac{1}{4} \abs{z}^2}}{|z|^{n+s}}\frac{\partial}{ \partial x_j}\left(\sinh\left(\frac{\langle z'', x''\rangle}{2}\right)\sinh\left(\frac{\langle z_{n-1}, x_{n-1}\rangle}{2}\right)\right)\sinh\left(\frac{ z_na}{2}\right)dz= 0.
\end{split}
\end{equation}
Since $\frac{e^{- \frac{1}{4} \abs{z}^2}}{|z|^{n+s}}\cosh\left(\frac{\langle z_{n-1}, x_{n-1}\rangle}{2}\right)z_{n-1}$ is odd in $z_{n-1}$, we deduce
$$D=\int_0^\infty \int_{\mathbb R^{n-2}} \int_{\mathbb R} \frac{e^{- \frac{1}{4} \abs{z}^2}}{|z|^{n+s}}\sinh\left(\frac{\langle z'', x''\rangle}{2}\right)\cosh\left(\frac{\langle z_{n-1}, x_{n-1}\rangle}{2}\right)z_{n-1}\sinh\left(\frac{ z_na}{2}\right)dz=0.$$
Plugging this information in \eqref{minimizza6}, we get that for every $x \in \partial E$ it holds
\begin{equation}\label{minimizza7}
\begin{split}
\int_{\{ z_n>0\}} \frac{e^{- \frac{1}{4} \abs{z}^2}}{|z|^{n+s}}\cosh\left(\frac{\langle z'', x''\rangle}{2}\right)\sinh\left(\frac{\langle z_{n-1}, x_{n-1}\rangle}{2}\right)\sinh\left(\frac{ z_na}{2}\right)z_{n-1}dz=0.
\end{split}
\end{equation}
We denote
$$C(z_n,x):=\int_{\mathbb R^{n-1}} \frac{e^{- \frac{1}{4} \abs{(z',z_n)}^2}}{|(z',z_n)|^{n+s}}\cosh\left(\frac{\langle z'', x''\rangle}{2}\right)\sinh\left(\frac{\langle z_{n-1}, x_{n-1}\rangle}{2}\right)z_{n-1}dz''dz_{n-1},$$
and we observe that if $x_{n-1}\neq 0$, then $C(z_n,x)\neq 0$ since the integrand is even in the variables $z''$ and $z_{n-1}$. Equation \eqref{minimizza7} then reads
\begin{equation}\label{minimizza8}
\int_0^\infty C(z_n,x)\sinh\left(\frac{ z_na}{2}\right)dz_n=0, \qquad \forall x \in \partial E,
\end{equation}
and since for every $z_n>0$
$$\sinh\left(\frac{ z_na}{2}\right)\left\{ \begin{array}{ccc}>0 & \text{if } a>0\\ =0 & \text{if } a=0\\ <0 & \text{if } a<0 \end{array}\right.,$$
equation \eqref{minimizza8} can hold if and only if $a=0$.
\end{proof}

\section{Acknowledgements}
The first author has been partially supported by the NSF DMS Grant No.~1906451.
The second author was partially supported by the Academy of Finland grant 314227.

{\small \noindent
Antonio De Rosa
\\
Department of Mathematics, University of Maryland, 4176 Campus Dr, College Park, MD 2074, USA
\\
\texttt{anderosa@umd.edu}
}

\bigskip

{\small \noindent
Domenico Angelo La Manna \\
Department of Mathematics and Statistics, P.O.\ Box 35 (MaD), FI-40014, University of Jyv\"askyl\"a, Finland.\\
\texttt{domenicolamanna@hotmail.it}
}

\begin{thebibliography}{2}
\bibitem{ADPM}\textsc{L. Ambrosio, G. De Philippis, L. Martinazzi,} \emph{Gamma-convergence of nonlocal perimeter functionals}, Manuscripta Math. \textbf{134} (2011), 377-403.
\bibitem{AFP}\textsc{L. Ambrosio, N. Fusco, D. Pallara,} \emph{Functions of bounded variation and free discontinuity problems}, Oxford Mathematical Monographs. The Clarendon Press, Oxford University Press, New York, 2000.
\bibitem{BBJ} {\sc M. Barchiesi, A. Brancolini , V. Julin}, {\em Sharp dimension free quantitative estimates for the Gaussian isoperimetric inequality.} arXiv:1409.2106v1 (2015), Ann. Probab., \textbf{45} (2017), 668--697.
\bibitem{Borell} {\sc C. Borell}, {\em The Brunn-Minkowski inequality in Gauss space.} Invent Math. \textbf{30}, (1975) 207--216. 
\bibitem{BBM}\textsc{J. Bourgain, H. Br\'ezis, P. Mironescu,} \emph{Another look at Sobolev spaces,} in Optimal Control and Partial Differential Equations (J. L. Menaldi, E. Rofman and A. Sulem, eds.), IOS Press (2001), 439-455.
\bibitem{CRS}\textsc{L. Caffarelli, J.-M. Roquejoffre, O. Savin,}
\emph{Non-local minimal surfaces}, preprint (2009).
\bibitem{CV}\textsc{L. Caffarelli, E. Valdinoci,}
\emph{Regularity properties of nonlocal minimal surfaces via limiting arguments}, preprint (2009).
\bibitem{DM}\textsc{G. Dal Maso,} \emph{An introduction to
$\Gamma$-convergence}, Birkh\"auser, 1993.
\bibitem{DG}\textsc{E. De Giorgi,} \emph{Nuovi teoremi relativi alle misure $ (r-1) $-dimensionali in uno spazio a $ r $ dimensioni,} Ricerche Mat., {\bf 4} (1955), 95-113
\bibitem{DGL}\textsc{E. De Giorgi, E. Letta,}\emph{Une notion g\'en\'erale de convergence faible pour des fonctions
croissantes d'ensemble,} Ann.  Scuola Norm.  Sup.  Pisa, {\bf (4)}
(1977), 61-99.
\bibitem{FM}\textsc{I. Fonseca, S. M\"uller,} \emph{Quasi-convex integrands and lower semicontinuity in $L^1$}, SIAM J. Math. Anal. \textbf{23} (1992), 1081-1098.

\bibitem{FFMMM}\textsc{A.Figalli, N. Fusco, F. Maggi, V. Millot ,M. Morini}\emph{ Isoperimetry and stability properties of balls with respect to nonlocal energies }, Comm. Math. Phis., 2014
\bibitem{giu}\textsc{E. Giusti,} \emph{Minimal surfaces and functions of bounded variation}, Monographs in mathematics, Brickhauser, Basel 1984.
\bibitem{L} {\sc D.A. La Manna},  {\em Local minimality of the ball for the Gaussian perimeter}, Advances in Calculus of Variations, {\bf (12)} (2019), 193–210.


\bibitem{Ma} \textsc{V. Maz'ya, T. Shaposhnikova,} \emph{Erratum to: ``On the Bourgain, Brezis, and Mironescu theorem concerning limiting embeddings of fractional Sobolev spaces''}, J. Funct. Anal. \textbf{201} (2003), 298-300.
		
\bibitem{M}\textsc{A.P. Morse,} \emph{Perfect blankets,} Trans. Amer. Math. Soc., {\bf 61} (1947), 418-422.

\bibitem {NPS}\textsc{M. Novaga, D. Pallara, Y. Sire}, \emph{A fractional isoperimetric problem in the Wiener space},J. Anal. Math., \textbf{134} (2018), 787-800.

\bibitem{V}\textsc{A. Visintin,} \emph{Generalized coarea formula and fractal sets,} Japan J. Indust. Appl. Math., {\bf 8} (1991), 175-201.
\end{thebibliography}
\end{document}